\documentclass{birkjour}
 \newtheorem{thm}{Theorem}[section]

 \newtheorem{prop}[thm]{Proposition}
 \theoremstyle{definition}
 
 \theoremstyle{remark}

 \numberwithin{equation}{section}

\allowdisplaybreaks
\usepackage{hyperref}
\usepackage{bm}
\usepackage[all]{xy}
\def\bmlambda{{\bm\lambda}}
\def\sumb{{\sum_{1\leq k_1{\tiny\begin{array}{ll}<k_2<\cdots<k_{r-1}&\!\!\!\!\!\!\!\!\!<\\
\leq\ell_1\leq\cdots\leq\ell_m&\!\!\!\!\!\!\!\!\!\leq\end{array}}k_r}}}
\def\sumu{{\sum_{1\leq k_1<\cdots<k_r\geq\ell_m\geq\cdots\geq\ell_1\geq 1}}}
\def\suml{{\sum_{k_r>\cdots>k_1\leq\ell_1\leq\cdots\leq\ell_m\atop k_1\geq 1}}}
\newcommand{\ov}[1]{\overline{#1}}
\newcommand{\ZU}[2]{%
Z_U\left(\begin{matrix}{#1}\\%
{#2}\end{matrix}\right)}
\newcommand{\ZL}[2]{%
Z_L\left(\begin{matrix}{#1}\\%
{#2}\end{matrix}\right)}
\newcommand{\ZB}[2]{%
Z_B\left(\begin{matrix}{#1}\\%
{#2}\end{matrix}\right)}

\begin{document}

%
%
%
%
%
%
%
%
%

\title[On Three General Forms of Multiple Zeta(-star) Values]
 {On Three General Forms of Multiple Zeta(-star) Values}

\author[K.-W.~Chen]{Kwang-Wu Chen}

\address{%
Department of Mathematics\\
University of Taipei\\
Taipei 10048\\
Taiwan}

\email{kwchen@uTaipei.edu.tw}

\thanks{The first author was funded by the Ministry of Science and Technology, 
Taiwan, R.O.C. grant number MOST 110-2115-M-845-001.}
\author[M.~Eie]{Minking Eie}
\address{Department of Mathematics\br
National Chung Cheng University\br
Chia-Yi 62145\br
Taiwan}
\email{minkingeie@gmail.com}


\date{February 8, 2022}
\dedicatory{February 08, 2022. Version 07.}

\begin{abstract}
In this paper, we investigate three general forms of multiple zeta(-star) values. 
We use these values to give three new sum formulas for multiple zeta(-star)
values with height $\leq 2$ and the evaluation of $\zeta^\star(\{1\}^m,\{2\}^{n+1})$. 
We also give a new proof of sum formula of multiple zeta values.
\end{abstract}

\maketitle
\section{Introduction}
For an $r$-tuple $\boldsymbol{\alpha} = (\alpha_{1}, \alpha_{2}, \ldots, \alpha_{r})$ 
of positive integers with $\alpha_{r} \geq 2$, a multiple zeta value $\zeta(\bm\alpha)$
and a multiple zeta-star value $\zeta^\star(\bm\alpha)$
are defined to be \cite{E09, E13, Ohno2005}

\begin{align*}
  \zeta(\boldsymbol{\alpha})
  &= \sum_{1 \leq k_{1} < k_{2} < \cdots < k_{r}} k_{1}^{-\alpha_{1}}
    k_{2}^{-\alpha_{2}} \cdots k_{r}^{-\alpha_{r}},
\quad\mbox{and}\\
  \zeta^{\star}(\boldsymbol{\alpha})
  &= \sum_{1 \leq k_{1} \leq k_{2} \leq \cdots \leq k_{r}} k_{1}^{-\alpha_{1}}
    k_{2}^{-\alpha_{2}} \cdots k_{r}^{-\alpha_{r}}.
\end{align*}

\noindent We denote the parameters $w(\bm\alpha)=|\bm\alpha| 
= \alpha_{1} + \alpha_{2} + \cdots + \alpha_{r}$,
$d(\bm\alpha)=r$, and $h(\bm\alpha)=\#\{i\mid \alpha_i>1, 1\leq i\leq r\}$,
called respectively the weight, the depth, and the height of $\bm\alpha$
(or of $\zeta(\bm\alpha)$, or of $\zeta^\star(\bm\alpha)$).
For any multiple zeta value $\zeta(\bm\alpha)$, 
we put a bar on top of $\alpha_j$ ($j=1,2,\ldots,r$) if there is a sign
$(-1)^{k_j}$ appearing in the numerator of its summation \cite{BBB1997, Xu2019}. 
For example,

$$
\zeta(\ov{\alpha_1},\ov{\alpha_2},\alpha_3,\ldots,\alpha_r)
=\sum_{1\leq k_1<k_2<\cdots<k_r}
\frac{(-1)^{k_1+k_2}}
{k_1^{\alpha_1}k_2^{\alpha_2}k_3^{\alpha_3}\cdots k_r^{\alpha_r}}.
$$

To find and study the relations among multiple zeta(-star) values 
is recognized as an important problem. Euler discovered a relation which 
connects double zeta values and a Riemann zeta value. 
In 1997, Granville \cite{Gran1997} gave a generalization of Euler's sum formula
to a general depth $r$, and that is called the sum formula. For positive integers
$r$ and $k$ with $k>r$, we have 

\begin{equation}\label{eq.sum}
\sum_{|\bm\alpha|=k}\zeta(\alpha_1,\ldots,\alpha_r)=\zeta(k).
\end{equation}

There is a corresponding sum formula for multiple zeta-star values

\begin{equation}\label{eq.sumstar}
\sum_{|\bm\alpha|=k}\zeta^\star(\alpha_1,\ldots,\alpha_r)
=\binom{k-1}{r-1}\zeta(k).
\end{equation}

So far, there are not many relations among multiple zeta-star values.
Aoki and Ohno \cite{AO2005} proved that for fixed weight $k$ 
and height $s$, we have

\begin{equation}\label{eq.fixwh}
\sum_{\bm\alpha\in I_0(k,s)}\zeta^\star(\bm\alpha)
=2\binom{k-1}{2s-1}(1-2^{1-k})\zeta(k),
\end{equation}

\noindent where $I_0(k,s)$ is the set of admissible multi-indices $\bm\alpha$
with weight $k$ and height $s$. 
In this paper, we investigate the following general forms of multiple zeta(-star) values:

\begin{align*}
&\ZU{\alpha_1,\alpha_2,\ldots,\alpha_r}{\beta_1,\beta_2,\ldots,\beta_m}\\
&:=\!\!\!\!\!\!
\sum_{1\leq k_1<k_2<\cdots<k_r}k_1^{-\alpha_1}k_2^{-\alpha_2}\cdots k_r^{-\alpha_r}
\sum_{1\leq\ell_1\leq\ell_2\leq\cdots\leq\ell_m\leq k_r}
\ell_1^{-\beta_1}\ell_2^{-\beta_2}\cdots\ell_m^{-\beta_m},\\
&\ZL{\alpha_1,\alpha_2,\ldots,\alpha_r}{\beta_1,\beta_2,\ldots,\beta_m}\\
&:=\!\!\!\!\!\!
\sum_{1\leq k_1<k_2<\cdots<k_r}k_1^{-\alpha_1}k_2^{-\alpha_2}\cdots k_r^{-\alpha_r}
\sum_{k_1\leq\ell_1\leq\ell_2\leq\cdots\leq\ell_m}
\ell_1^{-\beta_1}\ell_2^{-\beta_2}\cdots\ell_m^{-\beta_m},\\
&\ZB{\alpha_1,\alpha_2,\ldots,\alpha_r}{\beta_1,\beta_2,\ldots,\beta_m}\\
&:=\!\!\!\!\!\!
\sum_{1\leq k_1<k_2<\cdots<k_r}k_1^{-\alpha_1}k_2^{-\alpha_2}\cdots k_r^{-\alpha_r}
\sum_{k_1\leq\ell_1\leq\ell_2\leq\cdots\leq\ell_m\leq k_r}
\ell_1^{-\beta_1}\ell_2^{-\beta_2}\cdots\ell_m^{-\beta_m}.
\end{align*}

Kaneko and Yamamoto \cite{KY2018} conjecture that 
the following identity of function $Z_U$ 
can give all linear relations of multiple zeta values over $\mathbb Q$. 

\[
I\left(\ \begin{xy}
{(6,3)*++[o][F]{{\bm\alpha}}="k"},
{(15,-4)*++[F]{{\bm\beta}}="l"},
{(0,-3) \ar @{{*}-} "k"}, 
{"k" \ar @{-} "l"}, 
{(15,-4) \ar @{-} "l"}, 
\end{xy}
\ \right)=\ZU{\bm\alpha}{\bm\beta}, 
\]

\noindent where the left-hand side of the above identity 
is an integral associated with a $2$-posets 
which was introduced by Yamamoto \cite{Yama2017}.

In this paper, we first use the function $Z_B$ to give two new sum formulas
for multiple zeta-star values. 

\begin{thm}\label{thm.01}
For integers $m\geq 0$ and $n\geq 2$, we have 

\begin{equation}\label{eq.two}
\sum_{\alpha_1+\alpha_2=n\atop \alpha_1,\alpha_2\geq 1}\zeta^\star(\alpha_1,\{1\}^m,\alpha_2+1)
=(m+n)\zeta(m+n+1).
\end{equation}
\end{thm}

For brevity, we hereafter use the lowercase English letters 
and lowercase Greek letters in the summation, with or without subscripts, 
to denote the nonnegative 
and positive integers, unless otherwise specified. For instance,

\begin{align*}
\sum_{\alpha_1+\alpha_2=5}\zeta^\star(\alpha_1,1,\alpha_2+1)
&=\zeta^\star(1,1,5)+\zeta^\star(2,1,4)+\zeta^\star(3,1,3)+\zeta^\star(4,1,2) \\
  \sum_{a+b=3} (-1)^{b} \zeta(\{1\}^{a},b+2)
&= -\zeta(5) + \zeta(1,4) - \zeta(1,1,3) + \zeta(1,1,1,2).
\end{align*}

We derive from (\ref{eq.two}) to get a new sum formula 
for multiple zeta-star values with height two in the following theorem.

\begin{thm}\label{thm.02}
For any non-negative integer $p$, we have

\begin{equation}\label{eq.sumh2}
\sum_{a+b+m=p}\zeta^\star(a+2,\{1\}^m,b+2)
=\left(p^2+3p+1+\frac{p+3}{2^{p+2}}\right)\zeta(p+4).
\end{equation}
\end{thm}

Secondly, we use the function $Z_L$ to give a new weighted sum formula for 
multiple zeta values with height two.

\begin{thm}\label{thm.03}
For any non-negative integer $p$, we have

\begin{align}\label{eq.nsumh2}
\sum_{a+b+m=p}&(-1)^m\zeta(a+2,\{1\}^m,b+2)\\
&=2\left[(-1)^p-1\right]\zeta(p+2,\overline{2})
+\left(1-\frac1{2^{p+2}}\right)\zeta(p+4).\nonumber
\end{align}

In particular, if $p=2q$, then 

\begin{equation}\label{eq.evenh2}
\sum_{a+b+m=2q}(-1)^m\zeta(a+2,\{1\}^m,b+2)
=\left(1-\frac1{2^{2q+2}}\right)\zeta(2q+4).
\end{equation}

\end{thm}

Thirdly, we use the function $Z_U$ to give an evaluation of the multiple zeta-star values
$\zeta^\star(\{1\}^m,\{2\}^{n+1})$.

\begin{thm}\label{thm.4}
For nonnegative integers $m,n$, we have

\begin{align}\label{eq.1m2n}
\zeta^\star(\{1\}^m,\{2\}^{n+1})
&=\sum_{p+r=n}(-1)^r\zeta^\star(\{2\}^p)\\
&\qquad\times\sum_{|\bm d|=m}\zeta(d_1+2,\ldots,d_{r+1}+2)
\prod^{r+1}_{j=1}(d_j+1).\nonumber
\end{align}
\end{thm}

As applications of Theorem \ref{thm.4}, we give some interesting formulas. 
For any nonnegative integer $m$, we have

\begin{align*}
&\sum_{a+b=m}(a+1)(b+1)\zeta(a+2,b+2)\\
&\quad=(m+1)\zeta(2)\zeta(m+2)+(m+1)\zeta(2,m+2)-2\zeta(m+4)-2\zeta(m+1,3).
\end{align*}

and 

\[
\sum_{a+b=m-1}\zeta(2a+3)\zeta(2b+3)
=\frac{2m+1}2\zeta(2m+4)-2\zeta(1,2m+3).
\]

Moreover, we use the functions $Z_B$ and $Z_U$ to give 
antoher proof of sum formula of multiple zeta values in the final section.

\section{Preliminaries}
\subsection{Some Known Results of $Z_B$ and $Z_U$}
Here we list some preliminaries about symmetric functions.
For more details, we refer the reader to \cite{Mac1995, Stan1999}. 

A partition $\bmlambda$ of a nonnegative integer $n$ is a sequence 
$(\lambda_1,\ldots,\lambda_k)$ satisfying 
$\lambda_1\geq\cdots\geq\lambda_k$ and 
$|\bmlambda|=\sum^k_{i=1}\lambda_i=n$.
If $|\bmlambda|=n$, then we say that $\bmlambda$ is a partition of $n$,
and write $\bmlambda\vdash n$.
For a partition $\bmlambda$, we define 
$m_i=m_i(\bmlambda):=\#\{j\,:\,\lambda_j=i\}$. We write a partition $\bmlambda$
in the form of $\langle 1^{m_1}2^{m_2}\cdots\rangle$ and define

$$
\mu_{\bmlambda} = 1^{m_1}m_1!2^{m_2}m_2!\cdots.
$$

\noindent Let $\{x_1,x_2,\ldots\}$ be a set of indeterminates.
The complete homogeneous symmetric functions $H_{\bmlambda}$
is defined by

\begin{eqnarray*}
H_n&=&\sum_{i_1\leq\cdots\leq i_n}x_{i_1}\cdots x_{i_n}, \quad n\geq 1
\quad\mbox{(with $H_0=1$)}\\
H_\bmlambda&=&H_{\lambda_1}H_{\lambda_2}\cdots, \quad
\mbox{if }\bmlambda=(\lambda_1,\lambda_2,\ldots).
\end{eqnarray*}

\noindent The power sum symmetric functions $P_\bmlambda$ is defined by

\begin{eqnarray*}
P_n&=&\sum_{i}x_i^n, \quad n\geq 1
\quad\mbox{(with $P_0=1$)}\\
P_\bmlambda&=&P_{\lambda_1}P_{\lambda_2}\cdots, \quad
\mbox{if }\bmlambda=(\lambda_1,\lambda_2,\ldots).
\end{eqnarray*}

\noindent We can express the complete homogeneous symmetric functions $H_n$
in terms of the $P_\bmlambda$ \cite[(7.22)]{Stan1999}: 

$$
H_n=\sum_{\bmlambda\vdash n}\mu_\bmlambda^{-1}P_\bmlambda,
$$

Let $k_1\leq k_2\leq\cdots\leq k_{r+1}$ be an order sequence of $r+1$ positive integers.
We set $x_\ell=\ell^{-1}$, for $k_1\leq \ell\leq k_{r+1}$, and 
$x_\ell=0$, otherwise. Then the corresponding power-sum 
and complete homogeneous symmetric 
polynomials are 

\begin{align*}
p_m &=\sum_{k_1\leq j\leq k_{r+1}}j^{-m},\\
h_m &=\sum_{k_1\leq \ell_1\leq\ell_2\leq \cdots\leq\ell_m\leq k_{r+1}}
\frac1{\ell_1\ell_2\cdots\ell_m}
=\sum_{\bmlambda\vdash n}\mu_\bmlambda^{-1}p_\bmlambda.
\end{align*}

We reformulate a generalized duality theorem that 
appear in \cite{CCE2016} as the form we need.

\begin{prop}\cite[Theorem II]{CCE2016}\label{prop.163}
Let $\zeta(\alpha_1, \ldots, \alpha_q, \alpha_{q+1}+1)$ be 
a multiple zeta value of depth $q+1$ and weight $q+r+2$, with the dual
$\zeta(\beta_1,\ldots,\beta_r,\beta_{r+1}+1)$.
Then for any nonnegtaive integer $m$, we express

$$
\ZB{\beta_1,\ldots,\beta_r,\beta_{r+1}+1}{\{1\}^m}
$$

\noindent as

\begin{equation}\label{eq.07}
\sum_{|\bm{d}|=m}\prod_{j=1}^{q+1}\binom{\alpha_j+d_j-1}{d_j}
\zeta(\alpha_1+d_1, \ldots, 
\alpha_q+d_q, \alpha_{q+1}+d_{q+1}+1).
\end{equation}
\end{prop}

In 2015, the first author \cite{Chen2015} gave a expression of some special 
function $Z_U$. Simiarly, we can use the same trick as above to restate 
the expression we need.

\begin{prop}\cite[Theorem A]{Chen2015}\label{prop.257}
Let $q,r$ be a pair of positive integers, and $\zeta(\alpha_1,\ldots,\alpha_q)$
be a multiple zeta value of depth $q$ and weight $w=q+r$ with its dual being
$\zeta(\beta_1,\ldots,\beta_r)$ of depth $r$. Then for a nonnegative integer $m$,
we express

$$
\ZU{\beta_1,\ldots,\beta_r}{\{1\}^m}
$$

as

\begin{equation}\label{eq.209}
\sum_{|\bm{d}|=m}\zeta(\alpha_1+d_1,\ldots,\alpha_q+d_q)
\prod^q_{j=1}\binom{\alpha_j+d_j-1}{d_j}.
\end{equation}

\end{prop}

\subsection{Some Results of Sums of Multiple Zeta(-star) Values of Height One}
Moreover we need some information about 
the following sums of multiple zeta(-star) values of height one:

\begin{alignat*}{3}
  Z_-(n)&= \sum_{a+b=n} (-1)^{b} \zeta(\{1\}^{a},b+2),
  &\quad\textrm{ and } \quad
  &&Z_+(n)&= \sum_{a+b=n} \zeta(\{1\}^{a},b+2),\\
  Z^\star_-(n)&=\sum_{a+b=n} (-1)^b \zeta^\star(\{1\}^a,b+2),
  &\quad\textrm{ and } \quad
  &&Z^\star_+(n)&= \sum_{a+b=n} \zeta^\star(\{1\}^a,b+2).
\end{alignat*}

\noindent The values of $Z_-(n)$ are given by \cite{LM1995, CE2021}

\begin{equation}\label{eq.zn}
Z_-(n)=\left\{\begin{array}{ll}
2(1-2^{-2m-1})\zeta(2m+2), & \mbox{ if } n=2m\in\mathbb Z_{\geq 0},\\
0, & \mbox{otherwise}.
\end{array}\right.
\end{equation}

\noindent The values of $Z_+^\star(n)$ are \cite{Ohno2005}

\begin{equation}\label{eq.zsp}
Z_+^\star(n)=2(n+1)(1-2^{-n-1})\zeta(n+2).
\end{equation}

\begin{prop}\cite[Corollary 5.4]{CE2021}\label{prop.333}
For any nonnegative integer $p$, we have

\begin{equation}\label{eq.10}
\sum_{m+n=p}Z_{-}(m)Z^\star_{+}(n)
=2((-1)^p-1)\zeta(p+2,\overline{2})+(p+2)(p+1+2^{-p-2})\zeta(p+4).
\end{equation}
\end{prop}

\section{The Function $Z_B$}
\subsection{Definition of $Z_B$}
The function $Z_B$ is defined by 

\begin{align}\label{eq.defzb}
&\ZB{\alpha_1,\alpha_2,\ldots,\alpha_r}{\beta_1,\beta_2,\ldots,\beta_m}\\
&:=\sumb k_1^{-\alpha_1}k_2^{-\alpha_2}\cdots k_r^{-\alpha_r}
\ell_1^{-\beta_1}\ell_2^{-\beta_2}\cdots\ell_m^{-\beta_m},\nonumber
\end{align}

\noindent where $\alpha_r\geq 2$ for the sake of convergence.
The function $Z_B$ includes MZVs and MZSVs as special cases:

\begin{align}
\ZB{\alpha_1,\ldots,\alpha_r}{\emptyset}
&=\zeta(\alpha_1,\ldots,\alpha_r),\\
\ZB{\alpha_1,\alpha_2}{\beta_1,\ldots,\beta_m}
&=\zeta^\star(\alpha_1,\beta_1,\ldots,\beta_m,\alpha_2)
-\zeta(\alpha_1+\beta_1+\cdots+\beta_m+\alpha_2),\\
\ZB{\alpha_1}{\beta_1,\ldots,\beta_m}
&=\zeta(\alpha_1+\beta_1+\cdots+\beta_m).   \label{eq.b17}
\end{align}

\subsection{Proof of Theorem \ref{thm.01}}
We use the special case $r=1$, $q\in\mathbb Z_{\geq 0}$ 
in Proposition \ref{prop.163}.
That is, $\zeta(\alpha_1$, $\ldots$, $\alpha_q$, $\alpha_{q+1}+1)$ with its weight
$q+3=\alpha_1+\cdots+\alpha_{q+1}+1$, and its dual is 
$\zeta(\beta_1,\beta_2+1)$. And this implies that there exists
an unique $j$, $1\leq j\leq q+1$, $\alpha_j=2$, and $\alpha_i=1$, for all $i\neq j$,
we write $\bm\alpha_j=(\alpha_1,\ldots,\alpha_{q+1})$.
Therefore, we apply (\ref{eq.07}) in the following

\begin{align*}
&\zeta^\star(\beta_1,\{1\}^m,\beta_2+1)-\zeta(\beta_1+\beta_2+1+m)\\
&= \sum_{1\leq k_1<k_2}k_1^{-\beta_1}k_2^{-\beta_2-1}
\sum_{k_1\leq\ell_1\leq\cdots\leq\ell_m\leq k_2}
\frac1{\ell_1\ell_2\cdots\ell_m} =\ZB{\beta_1,\beta_2+1}{\{1\}^m}\\
&=\sum_{|\bm{d}|=m}(d_j+1)\zeta(d_1+1,\ldots,d_{j-1}+1,d_j+2,
d_{j+1}+1,\ldots,d_q+1,d_{q+2}+2) \\
&=\sum_{|\bm{\gamma}|=m+q+2}(\gamma_j-1)
\zeta(\gamma_1,\ldots,\gamma_q,\gamma_{q+1}+1).
\end{align*}

\noindent We add for all possible such parameters $\bm\alpha_j$,
that is we add for $1\leq j\leq q+1$ with $\bm\alpha_j$ 
on the right-hand side of the above equation 
and their corresponding dual $\bm\beta$ with $\beta_1+\beta_2=q+2$ on the left-hand side.
The above identity becomes

\begin{align*}
&\sum_{|\bm{\beta}|=q+2}\zeta^\star(\beta_1,\{1\}^m,\beta_2+1)
-(q+1)\zeta(q+3+m) \\
&\quad=\sum_{|\bm{\gamma}|=m+q+2}\zeta(\gamma_1,\ldots,\gamma_q,\gamma_{q+1}+1)
\sum^{q+1}_{j=1}(\gamma_j-1) 
=(m+1)\zeta(m+q+3).
\end{align*}

\noindent In the last equality we use the sum formula (\ref{eq.sum}) 
for multiple zeta values. Therefore, 

$$
\sum_{|\bm{\beta}|=q+2}\zeta^\star(\beta_1,\{1\}^m,\beta_2+1)
=(m+q+2)\zeta(m+q+3).
$$

\noindent This gives us the desired result. \qed

\subsection{Proof of Theorem \ref{thm.02}}
We reform the following sum as

\begin{align*}
\sum_{a+b=n}&\zeta^\star(a+2,\{1\}^m,b+2) \\
&=\sum_{\alpha_1+\alpha_2=n+2}\zeta^\star(\alpha_1+1,\{1\}^m,\alpha_2+1) \\
&=\sum_{\beta_1+\beta_2=n+3}
\zeta^\star(\beta_1,\{1\}^m,\beta_2+1)-\zeta^\star(\{1\}^{m+1},n+3).
\end{align*}

\noindent We have evaluated the above sum in Theorem \ref{thm.01}. Thus,

$$
\sum_{a+b=n}\zeta^\star(a+2,\{1\}^m,b+2) 
=(m+n+3)\zeta(m+n+4)-\zeta^\star(\{1\}^{m+1},n+3).
$$

\noindent Here we now evaluate our desired sum

\begin{align*}
\sum_{a+b+m=p}&\zeta^\star(a+2,\{1\}^m,b+2) \\
&=\sum_{m+n=p}\sum_{a+b=n}\zeta^\star(a+2,\{1\}^m,b+2)\\
&=(p+1)(p+3)\zeta(p+4)-\sum_{m+n=p}\zeta^\star(\{1\}^{m+1},n+3)\\
&=(p+1)(p+3)\zeta(p+4)+\zeta(p+4)+\zeta^\star(\{1\}^{p+2},2)\\
&\qquad\qquad\qquad\qquad\qquad\quad
-\sum_{a+b=p+2}\zeta^\star(\{1\}^a,b+2).
\end{align*}

\noindent The last sum $Z^\star_+(p+2)$ \cite{Ohno2005,CE2021} is stated in (\ref{eq.zsp})
and 

\[
\zeta^\star(\{1\}^{p+2},2)=(p+3)\zeta(p+4),
\] 

\noindent thus we conclude the result. \qed

\section{The Function $Z_L$}
\subsection{Definition of $Z_L$}
The function $Z_L$ is defined by

\begin{align}\label{eq.defzl}
&\ZL{\alpha_1,\alpha_2,\ldots,\alpha_r}{\beta_1,\beta_2,\ldots,\beta_m}\\
&:=\suml
k_1^{-\alpha_1}k_2^{-\alpha_2}\cdots k_r^{-\alpha_r}
\ell_1^{-\beta_1}\ell_2^{-\beta_2}\cdots\ell_m^{-\beta_m},\nonumber
\end{align}

\noindent where $\alpha_r\geq 2$ and $\beta_m\geq 2$ for the sake of convergence.
The function $Z_L$ includes MZVs and MZSVs as special cases:

\begin{align}\label{eq.zl0}
\ZL{\alpha_1,\alpha_2,\ldots,\alpha_r}{\emptyset}
&=\zeta(\alpha_1,\alpha_2,\ldots,\alpha_r),\\
\ZL{\alpha_1}{\beta_1,\beta_2,\ldots,\beta_m}
&=\zeta^\star(\alpha_1,\beta_1,\ldots,\beta_m).
\end{align}

For $\alpha_r\geq 2$ and $\beta_m\geq 2$, we multiply

\begin{align*}
\zeta(\alpha_1,\ldots,\alpha_r)&=\sum_{1\leq k_1<\cdots<k_r}
k_1^{-\alpha_1}\cdots k_r^{-\alpha_r},\\
\zeta^\star(\beta_1,\ldots,\beta_m)&=\sum_{1\leq\ell_1\leq\cdots\leq\ell_m}
\ell_1^{-\beta_1}\cdots\ell_m^{-\beta_m}
\end{align*}

\noindent together. We divide the resulted summation into 
two situations $k_1\leq\ell_1$ and $\ell_1<k_1$. Then 

\begin{align}\label{eq.414}
&\zeta(\alpha_1,\alpha_2,\ldots,\alpha_r)\
\zeta^\star(\beta_1,\beta_2,\ldots,\beta_m)\\
&=\ZL{\alpha_1,\alpha_2,\ldots,\alpha_r}{\beta_1,\beta_2,\ldots,\beta_m}+
\ZL{\beta_1,\alpha_1,\ldots,\alpha_r}{\beta_2,\beta_3,\ldots,\beta_m}.\nonumber
\end{align}

The above formula can give an identity with MZVs and MZSVs.

\begin{prop}\label{prop.428}
For an $r$-tuple $\bm\alpha=(\alpha_1,\alpha_2,\ldots,\alpha_r)$ of positive integers
with $\alpha_1\geq 2$, $\alpha_r\geq 2$, we have

\begin{align}\label{eq.415}
&\zeta(\alpha_1,\ldots,\alpha_r)+(-1)^r\zeta^\star(\alpha_r,\ldots,\alpha_1)\\
&=\sum^{r-1}_{k=1}(-1)^{k+1}\zeta^\star(\alpha_k,\ldots,\alpha_1)
\zeta(\alpha_{k+1},\ldots,\alpha_r).\nonumber
\end{align}
\end{prop}

\begin{proof}
Using (\ref{eq.414}), we have for $1\leq k\leq r$

$$
\ZL{\alpha_k,\ldots,\alpha_r}{\alpha_{k-1},\ldots,\alpha_1}
=\zeta^\star(\alpha_k,\ldots,\alpha_1)\zeta(\alpha_{k+1},\ldots,\alpha_r)
-\ZL{\alpha_{k+1},\ldots,\alpha_r}{\alpha_k,\ldots,\alpha_1}.
$$

\noindent Then we use the induction on $k$, we obtain that 
$\zeta(\alpha_1,\ldots,\alpha_r)$ is equal to 

$$
(-1)^k\ZL{\alpha_{k+1},\ldots,\alpha_r}{\alpha_k,\ldots,\alpha_1}
+\sum^k_{\ell=1}(-1)^{\ell-1}\zeta^\star(\alpha_\ell,\ldots,\alpha_1)
\zeta(\alpha_{\ell+1},\ldots,\alpha_r).
$$

The case $k=r-1$ is what we want.
\end{proof}

\subsection{Proof of Theorem \ref{thm.03}}
Applying (\ref{eq.415}) we rewrite

\begin{align*}
(-1)^m\zeta(c+2,\{1\}^m,d+2)&+\zeta^\star(d+2,\{1\}^m,c+2)\\
&=\sum_{k+\ell=m}(-1)^\ell\zeta^\star(\{1\}^k,c+2)\zeta(\{1\}^\ell,d+2).
\end{align*}

\noindent Therefore,

\begin{align*}
\sum_{c+d+m=p}&\left[(-1)^m\zeta(c+2,\{1\}^m,d+2)
+\zeta^\star(d+2,\{1\}^m,c+2)\right] \\
&=\sum_{c+d+k+\ell=p}(-1)^\ell\zeta^\star(\{1\}^k,c+2)\zeta(\{1\}^\ell,d+2).
\end{align*}

\noindent The above sum can be rewritten as 

\begin{align*}
\sum_{m+n=p}\sum_{k+c=m}&\zeta^\star(\{1\}^k,c+2)
\sum_{\ell+d=n}(-1)^\ell\zeta(\{1\}^\ell,d+2)\\
&=\sum_{m+n=p}(-1)^nZ^\star_+(m)Z_-(n).
\end{align*}

Since $Z_-(2n-1)=0$, for any positive integer $n$, the last sum 
is exactly (\ref{eq.10}). Thus,

\begin{align*}
\sum_{c+d+m=p}&\left[(-1)^m\zeta(c+2,\{1\}^m,d+2)
+\zeta^\star(d+2,\{1\}^m,c+2)\right] \\
&=2((-1)^p-1)\zeta(p+2,\overline{2})+(p+2)(p+1+2^{-p-2})\zeta(p+4).
\end{align*}

\noindent On the other hand, Theorem \ref{thm.02} gives 

$$
\sum_{c+d+m=p}\zeta^\star(d+2,\{1\}^m,c+2)
=\left(p^2+3p+1+\frac{p+3}{2^{p+2}}\right)\zeta(p+4).
$$

\noindent Combine these results together, we have the desired result. \qed

\vskip 0.5cm
Similarly, we use the same trick in the proof, that is, we apply
(\ref{eq.415}) to transform the following summation of multiple zeta(-star) values with height two
into a convolution of sums of multiple zeta(-star) values with height one:

\begin{align*}
&\sum_{c+d+m=p}\left[\zeta(c+2,\{1\}^m,d+2)
+(-1)^m\zeta^\star(d+2,\{1\}^m,c+2)\right] \\
&\qquad\qquad\qquad=\sum_{m+n=p}(-1)^mZ_+(m)Z_-^\star(n).
\end{align*}

\section{The Function $Z_U$}
\subsection{Definition of $Z_U$}
The function $Z_U$ is defined by 

\begin{align}\label{eq.defzu}
&\ZU{\alpha_1,\alpha_2,\ldots,\alpha_r}{\beta_1,\beta_2,\ldots,\beta_m}\\
&:=\sumu
k_1^{-\alpha_1}k_2^{-\alpha_2}\cdots k_r^{-\alpha_r}
\ell_1^{-\beta_1}\ell_2^{-\beta_2}\cdots\ell_m^{-\beta_m},\nonumber
\end{align}

\noindent where $\alpha_r\geq 2$ for the sake of convergence.
The function $Z_U$ includes MZVs and MZSVs as special cases:

\begin{align}
\ZU{\alpha_1,\ldots,\alpha_r}{\emptyset}
&=\zeta(\alpha_1,\ldots,\alpha_r),\\
\ZU{\alpha_1}{\beta_1,\ldots,\beta_m}
&=\zeta^\star(\beta_1,\ldots,\beta_m,\alpha_1).
\end{align}

For $\alpha_r\geq 2$ and $\beta_m\geq 2$, we multiply

\begin{align*}
\zeta(\alpha_1,\ldots,\alpha_r)&=\sum_{1\leq k_1<\cdots<k_r}
k_1^{-\alpha_1}\cdots k_r^{-\alpha_r},\\
\zeta^\star(\beta_1,\ldots,\beta_m)&=\sum_{1\leq\ell_1\leq\cdots\leq\ell_m}
\ell_1^{-\beta_1}\cdots\ell_m^{-\beta_m}
\end{align*}

\noindent together. We divide the resulted summation into 
two situations $k_r\geq\ell_m$ and $\ell_m>k_r$. Then 

\begin{align}\label{eq.516}
&\zeta(\alpha_1,\alpha_2,\ldots,\alpha_r)\
\zeta^\star(\beta_1,\beta_2,\ldots,\beta_m) \\
&=\ZU{\alpha_1,\alpha_2,\ldots,\alpha_r}{\beta_1,\beta_2,\ldots,\beta_m}+
\ZU{\alpha_1,\ldots,\alpha_r,\beta_m}{\beta_1,\ldots,\beta_{m-1}}. \nonumber
\end{align}

The above formula can give an identity which was proved by Zlobin in 2005.

\begin{prop}\cite[Theorem 3]{Zlobin2005}\label{prop.558}
For an $r$-tuple $\bm\alpha=(\alpha_1,\alpha_2,\ldots,\alpha_r)$ of positive integers
with $\alpha_i\geq 2$, for $1\leq i\leq r$, we have

\begin{align}\label{eq.517}
\zeta^\star(\alpha_r,\ldots,\alpha_1)&+(-1)^r\zeta(\alpha_1,\ldots,\alpha_r)\\
&=\sum^{r-1}_{k=1}(-1)^{k+1}
\zeta(\alpha_{1},\ldots,\alpha_{k})
\zeta^\star(\alpha_r,\ldots,\alpha_{k+1}).\nonumber
\end{align}
\end{prop}

\begin{proof}
Using (\ref{eq.516}), we have for $1\leq k\leq r$

$$
\ZU{\alpha_{1},\ldots,\alpha_{k-1}}{\alpha_r,\ldots,\alpha_k}
=\zeta^\star(\alpha_r,\ldots,\alpha_k)\zeta(\alpha_{1},\ldots,\alpha_{k-1})
-\ZU{\alpha_1,\ldots,\alpha_k}{\alpha_{r},\ldots,\alpha_{k+1}}.
$$

\noindent Then we use the induction on $k$, we obtain
that $\zeta^\star(\alpha_r,\ldots,\alpha_1)$ is equal to 

\begin{equation}\label{eq.zur}
(-1)^k\ZU{\alpha_{1},\ldots,\alpha_{k+1}}{\alpha_r,\ldots,\alpha_{k+2}}
+\sum^k_{\ell=1}(-1)^{\ell-1}
\zeta(\alpha_{1},\ldots,\alpha_{\ell})
\zeta^\star(\alpha_r,\ldots,\alpha_{\ell+1}).
\end{equation}

The case $k=r-1$ is what we want.
\end{proof}

\subsection{Proof of Theorem \ref{thm.4}}
Applying Proposition \ref{prop.257} to the multiple zeta value
$\zeta(\{2\}^{r+1})$, we have

\[
\ZU{\{2\}^{r+1}}{\{1\}^m}
=\sum_{|{\bm d}|=m}\zeta(d_1+2,\ldots,d_{r+1}+2)\prod^{r+1}_{j=1}(d_j+1).
\]

Using Eq.\,(\ref{eq.zur}) we can write $\zeta^\star(\{1\}^m,\{2\}^{r+1})$ as

\[
(-1)^{r}\ZU{\{2\}^{r+1}}{\{1\}^m}
+\sum^{r}_{k=1}(-1)^{k-1}\zeta(\{2\}^k)\zeta^\star(\{1\}^m,\{2\}^{r+1-k}).
\]

Combing these two equations together, we have

\begin{align}\label{eq.519}
\sum_{q+k=r}&(-1)^k\zeta(\{2\}^k)\zeta^\star(\{1\}^m,\{2\}^{q+1})\\
&=(-1)^{r}\sum_{|\bm d|=m}\zeta(d_1+2,\ldots,d_{r+1}+2)
\prod^{r+1}_{j=1}(d_j+1).\nonumber
\end{align}

Multiplying Eq.\,(\ref{eq.519}) with $\zeta^\star(\{2\}^p)$
and taking summation with $p+r=n$ for all nonnegative integer $p$.
Then  the left-hand side of the above equation becomes

\begin{align*}
&\sum_{p+q+k=n}(-1)^k\zeta(\{2\}^k)\zeta^\star(\{2\}^p)\zeta^\star(\{1\}^m,\{2\}^{q+1}) \\
&\quad=\sum_{s+q=n}\left(\sum_{k+p=s}(-1)^k\zeta(\{2\}^k)\zeta^\star(\{2\}^p)\right)
\zeta^\star(\{1\}^m,\{2\}^{q+1})\\
&\quad=\zeta^\star(\{1\}^m,\{2\}^{n+1}).
\end{align*}

We use the following known result in the last equation.

\begin{equation}\label{eq.one}
\sum_{a+b=n}(-1)^a\zeta(\{s\}^a)\zeta^\star(\{s\}^b)
=\left\{\begin{array}{ll}
1, & \mbox{ if }n=0,\\
0, & \mbox{ if }n\neq 0.
\end{array}\right.
\end{equation}

Therefore we can conclude Theorem \ref{thm.4}. \qed
\section{Some Applications of Theorem \ref{thm.4}}
\subsection{The Cases $n=0, 1$}
We use Eq.\,(\ref{eq.1m2n}) to give some formulas. For the case $n=0$, we have

\[
\zeta^\star(\{1\}^m,2)=(m+1)\zeta(m+2).
\]

This formula is a special case in \cite[Theorem 4]{Zlobin2005}.

For the case $n=1$ in Eq\,(\ref{eq.1m2n}), we have

\[
\zeta^\star(\{1\}^m,2,2)
=(m+1)\zeta(2)\zeta(m+2)-\sum_{a+b=m}(a+1)(b+1)\zeta(a+2,b+2).
\]

However, there is another expression of $\zeta^\star(\{1\}^m,2,2)$ 
given by \cite[Theorem 4]{Zlobin2005}:

\[
\zeta^\star(\{1\}^m,2,2)
=2\zeta(m+4)+2\zeta(m+1,3)-(m+1)\zeta(2,m+2).
\]

Combine these two identities, we obtain a weighted sum formula:

\begin{prop}
For any nonnegative integer $m$, we have

\begin{align*}
&\sum_{a+b=m}(a+1)(b+1)\zeta(a+2,b+2)\\
&\quad=(m+1)\zeta(2)\zeta(m+2)+(m+1)\zeta(2,m+2)-2\zeta(m+4)-2\zeta(m+1,3).
\end{align*}
\end{prop}

The basic stuffle product relation is 

\begin{equation}\label{eq.stuffle}
\zeta(a)\zeta(b)=\zeta(a,b)+\zeta(b,a)+\zeta(a+b).
\end{equation}

Using this equation we get

\begin{align*}
&\sum_{a+b=m}(a+1)(b+1)\zeta(a+2)\zeta(b+2)\\
&\quad=\frac{(m+1)(m+2)(m+3)}6\zeta(m+4)+2(m+1)\zeta(2)\zeta(m+2)\\
&\qquad+2(m+1)\zeta(2,m+2)-4\zeta(3)\zeta(m+1)+4\zeta(3,m+1).
\end{align*}

There is another expression given by the second author and Yang \cite{EY2015}:

\begin{align*}
&\frac12\sum_{a+b=m}(a+1)(b+1)\zeta(a+2)\zeta(b+2)\\
&\quad=\sum_{\alpha+\beta=m}2^\beta(\alpha+1)(\beta+1)\zeta(\alpha+2,\beta+2) \\
&\qquad+\sum_{\alpha+\beta=m}2^\beta(\beta+1)(\beta+2)\zeta(\alpha+1,\beta+3).
\end{align*}

\subsection{The Cases $m=1,2,3$}
If we let $m=1$ in Eq.\,(\ref{eq.1m2n}), we have

\[
\zeta^\star(1,\{2\}^{n+1})
=2\sum_{p+r=n}(-1)^r\zeta^\star(\{2\}^p)
\sum_{a+b=r}\zeta(\{2\}^a,3,\{2\}^b).
\]

The inner summation of the above sum is given by Zagier \cite{Zagier2012}:

\[
\sum_{a+b=r}\zeta(\{2\}^a,3,\{2\}^b)
=\sum_{a+b=r}(-1)^a\zeta(\{2\}^b)\zeta(2a+3).
\]

Therefore 

\begin{align*}
\zeta^\star(1,\{2\}^{n+1})
&=2\sum_{p+a+b=n}(-1)^b\zeta^\star(\{2\}^p)\zeta(\{2\}^b)\zeta(2a+3)\\
&=2\sum_{\ell+a=n}\left(\sum_{p+b=\ell}(-1)^b\zeta^\star(\{2\}^p)\zeta(\{2\}^b)\right)
\zeta(2a+3).
\end{align*}

We apply Eq\,(\ref{eq.one}) and get (see \cite{Zlobin2005})

\[
\zeta^\star(1,\{2\}^{n+1})
=2\zeta(2n+3).
\]

For the case $m=2$ of Eq.\,(\ref{eq.1m2n}), we have

\[
\zeta^\star(1,1,\{2\}^{n+1})
=\sum_{p+r=n}(-1)^r\zeta^\star(\{2\}^p)
\sum_{|\bm d|=2}\zeta(d_1+2,\ldots,d_{r+1}+2)\prod^{r+1}_{j=1}(d_j+1).
\]

In fact, the inner summation is

\[
3\sum_{a+b=r}\zeta(\{2\}^a,4,\{2\}^b)
+4\sum_{a+b+c=r-1}\zeta(\{2\}^a,3,\{2\}^b,3,\{2\}^c).
\]

Consider the generating functions of 
$\zeta(\{s\}^{a_0},r_1,\{s\}^{a_1},\ldots,r_n,\{s\}^{a_n})$, 
it is easy to get the following identity (ref. \cite[Proposition 6.5]{CE2021x})

\begin{align}\label{eq.rs}
\sum_{|\bm a|=m}&\zeta(\{s\}^{a_0},r_1,\{s\}^{a_1},
\ldots,r_n,\{s\}^{a_n})\\
&=\sum_{|\bm a|=m}(-1)^{m-a_0}
\zeta(\{s\}^{a_0}) \zeta(sa_1+r_1,\ldots,sa_n+r_n).\nonumber
\end{align}

Using Eq.\,(\ref{eq.rs}) and Eq.\,(\ref{eq.one}) we can simplify
the expression of $\zeta^\star(1$, $1$, $\{2\}^{n+1})$ as follows.

\[
3\zeta(2n+4)-4\sum_{a+b=n-1}\zeta(2a+3,2b+3).
\]

Using the stuffle product relation Eq.\,(\ref{eq.stuffle}), we have

\begin{equation}\label{eq.112n}
\zeta^\star(1,1,\{2\}^{n+1}) = (2n+3)\zeta(2n+4)
-2\sum_{a+b=n-1}\zeta(2a+3)\zeta(2b+3).
\end{equation}

Ohno and Zudilin gave another expression \cite[Lemma 5]{OZ2008}

\[
\zeta^\star(1,1,\{2\}^{n+1})
=4\zeta^\star(1,2n+3)-2\zeta(2n+4).
\]

Combing theses two expressions, we have the following proposition.

\begin{prop}
For any nonnegative integer $n$, we have

\begin{equation}
\sum_{a+b=n-1}\zeta(2a+3)\zeta(2b+3)
=\frac{2n+1}2\zeta(2n+4)-2\zeta(1,2n+3).
\end{equation}
\end{prop}

The second author and Yang \cite[Remark 4]{EY2015} gave a similar identity:

\begin{equation}
\sum_{a+b=n}\zeta(a+2)\zeta(b+2)
=(n+3)\zeta(n+4)-2\zeta(1,n+3).
\end{equation}

Similarly, we have the following identity for the case $m=3$.

\begin{align*}
\zeta^\star(\{1\}^3,\{2\}^{n+1}) &=\frac{2}3(2n+3)(n+2)\zeta(2n+5)\\
&\quad-6\sum_{a+b=n-1}\zeta(2a+4)\zeta(2b+3) \\
&\quad-4\sum_{a+b=n-2}(a+1)\zeta(2a+6)\zeta(2b+3)\\
&\quad+\frac43\sum_{a+b+c=n-2}\zeta(2a+3)\zeta(2b+3)\zeta(2c+3).
\end{align*}

\section{Another Proof of Sum Formula}
\subsection{Relations Between $Z_B$, $Z_U$, and $Z_L$}
We divide the summation in the definition of $Z_U$ into two cases:
$k_1\leq\ell_1$ and $\ell_1<k_1$. Then we have

\begin{equation}\label{eq.zutozb}
\ZU{\alpha_1,\ldots,\alpha_r}{\beta_1,\ldots,\beta_m}
=\ZB{\alpha_1,\ldots,\alpha_r}{\beta_1,\ldots,\beta_m}
+\ZB{\beta_1,\alpha_1,\ldots,\alpha_r}{\beta_2,\beta_3,\ldots,\beta_m}.
\end{equation}

We use this formula recursively for $p\geq 1$, we get

\begin{align}\label{eq.zubr}
&\ZB{\alpha_{m+1},\ldots,\alpha_r}{\alpha_m,\ldots,\alpha_1}\\
&=(-1)^p\ZB{\alpha_{m+1-p},\ldots,\alpha_r}{\alpha_{m-p},\ldots,\alpha_1}
+\sum^p_{k=1}(-1)^{k-1}
\ZU{\alpha_{m+2-k},\ldots,\alpha_r}{\alpha_{m+1-k},\ldots,\alpha_1}.\nonumber
\end{align}

If we set $p=m$ in Eq\,(\ref{eq.zubr}), then we have the following identity.
\begin{prop}\label{prop.zuzbzeta}
For $\alpha_r\geq 2$ and $1\leq m\leq r$, we have

\begin{align}\label{eq.zuzbzeta}
&\ZB{\alpha_{m+1},\ldots,\alpha_r}{\alpha_m,\ldots,\alpha_1}\\
&\quad=(-1)^m\zeta(\alpha_1,\ldots,\alpha_r)+\sum^m_{k=1}(-1)^{m-k}
\ZU{\alpha_{k+1},\ldots,\alpha_r}{\alpha_k,\ldots,\alpha_1}.\nonumber
\end{align}

\end{prop}

Similarly, the function $Z_L$ has the relations with $Z_B$:

\begin{align*}
\ZL{\alpha_1,\ldots,\alpha_r}{\beta_1,\ldots,\beta_m}
&=\ZB{\alpha_1,\ldots,\alpha_r}{\beta_1,\ldots,\beta_m}
+\ZB{\alpha_1,\ldots,\alpha_r,\beta_m}{\beta_1,\ldots,\beta_{m-1}},\\
\ZB{\alpha_r,\ldots,\alpha_{m+1}}{\alpha_1,\ldots,\alpha_{m}}
&=(-1)^p\ZB{\alpha_r,\ldots,\alpha_{m+1-p}}{\alpha_1,\ldots,\alpha_{m-p}}\\
&\quad+\sum^p_{k=1}(-1)^{k-1}
\ZL{\alpha_r,\ldots,\alpha_{m+2-k}}{\alpha_1,\ldots,\alpha_{m+1-k}},
\end{align*}
where $p\geq 1$. If we set $p=m$, then we have

\begin{align}
&\ZB{\alpha_r,\ldots,\alpha_{m+1}}{\alpha_1,\ldots,\alpha_{m}}\\
&\quad=(-1)^{m}\zeta(\alpha_r,\ldots,\alpha_1)+
\sum^{m}_{k=1}(-1)^{m-k}
\ZL{\alpha_r,\ldots,\alpha_{k+1}}{\alpha_1,\ldots,\alpha_{k}}.\nonumber
\end{align}

\subsection{Another Proof for Sum Formula of MZVs}
In order to give another proof of Eq\,(\ref{eq.sum}), we need 
the following result.

\begin{prop}\cite[Theorem A]{CE2018}
For any nonnegative integers $p,q$, and $n$, we have

\begin{align}\label{eq.zuint}
&\ZU{\{1\}^p,n+2}{\{1\}^q}\\
&\quad=\frac1{p!q!n!}\int_{E_2}\left(\log\frac1{1-t_1}\right)^p
\left(\log\frac1{1-t_2}\right)^q\left(\log\frac{t_2}{t_1}\right)^n
\frac{dt_1dt_2}{(1-t_1)t_2},\nonumber
\end{align}

\noindent where $E_2$ is the region $0<t_1<t_2<1$.
\end{prop}

From the basic special cases Eq.\,(\ref{eq.b17}) we know that 

$$
\zeta(m+n+2)=\ZB{n+2}{\{1\}^m}.
$$

Applying Eq.\,(\ref{eq.zuzbzeta}), this value of $Z_B$ can be written as

$$
\sum_{k=0}^m(-1)^{k}\ZU{\{1\}^k,n+2}{\{1\}^{m-k}}.
$$

Summing the corresponding integral representation of Eq\,(\ref{eq.zuint}), 
we have

$$
\frac1{m!n!}\int_{E_2}\left(\log\frac{1-t_1}{1-t_2}\right)^m
\left(\log\frac{t_2}{t_1}\right)^n\frac{dt_1dt_2}{(1-t_1)t_2}.
$$

This integral is equal to \cite{ELO2009}

$$
\sum_{|\bm\alpha|=m+n+1}\zeta(\alpha_1,\ldots,\alpha_m,\alpha_{m+1}+1).
$$

Therefore, we get the sum formula of multiple zeta values.



\end{document}